\documentclass[review]{elsarticle}

\journal{}

\usepackage[margin=1in]{geometry}  % set the margins to 1in on all sides
\usepackage{graphicx}              % to include figures
\usepackage{amsmath}               % great math stuff
\usepackage{amsfonts}              % for blackboard bold, etc
\usepackage{amsthm}                % better theorem environments
\usepackage{amssymb}
\newtheorem{thm}{Theorem}[section]
\newtheorem{lem}[thm]{Lemma}
\newtheorem{prop}[thm]{Proposition}

\DeclareMathOperator{\Aut}{Aut}
\DeclareMathOperator{\PGL}{PGL}
\DeclareMathOperator{\PSL}{PSL}
\DeclareMathOperator{\Cay}{Cay}
\DeclareMathOperator{\rate}{rate}
\DeclareMathOperator{\distance}{distance}

\usepackage{lipsum}
\makeatletter
\def\ps@pprintTitle{%
 \let\@oddhead\@empty
 \let\@evenhead\@empty
 \def\@oddfoot{}%
 \let\@evenfoot\@oddfoot}
\makeatother

\begin{document}

\begin{frontmatter}

\title{Symmetric Unique Neighbor Expanders and Good LDPC Codes}

\author{Oren Becker\corref{mycorrespondingauthor}}
\address{Institute of Mathematics, Hebrew University, Jerusalem 9190401 ISRAEL}
%%\cortext[mycorrespondingauthor]{Corresponding author}
\ead{oren.becker@mail.huji.ac.il}

\begin{abstract}
An infinite family of bounded-degree 'unique-neighbor' expanders was
constructed explicitly by Alon and Capalbo (2002).
We present an infinite family $\mathcal{F}$ of bounded-degree unique-neighbor
expanders with the additional property that every graph in the family
$\mathcal{F}$ is a \emph{Cayley graph}. This answers a question raised
by Tali Kaufman. Using the same methods, we show that the symmetric
LDPC codes constructed by Kaufman and Lubotzky (2012) are in fact symmetric
under a \emph{simply} transitive group action on coordinates.
\end{abstract}

\begin{keyword}
Unique neighbor expander \sep Error correcting code \sep Cayley graph
\end{keyword}

\end{frontmatter}

\section{Introduction}

An undirected graph $\Gamma=\left(V,E\right)$ is an \emph{$\left(\alpha,\epsilon\right)$-unique-neighbor
expander} if for every subset $X$ of $V$ such that $\left|X\right|\leq\alpha\left|V\right|$,
there are at least $\epsilon\left|X\right|$ vertices in $V\setminus X$
that are adjacent to \emph{exactly} one vertex in $X$. In \cite{alon2002explicit},
Alon and Capalbo construct families of bounded-degree \emph{$\left(\alpha,\epsilon\right)$-}unique-neighbor\emph{
}expanders for some positive $\alpha,\epsilon$. One of the families
constructed in \cite{alon2002explicit} is an infinite family $\mathcal{F}$ of
$6$-regular $\left(\alpha,\epsilon\right)$-unique-neighbor expanders.
This construction is based on a notion of a product of graphs, forming
a product graph $\Gamma$ from a $d$-regular graph $\Gamma'$ and
a graph $\Delta$ on $d$ vertices. It is shown in \cite{alon2002explicit}
that if $\Gamma'$ is an $8$-regular Ramanujan graph, and if $\Delta$
is a certain specific graph on $8$ vertices, then the product graph
$\Gamma$ is an $\left(\alpha,\epsilon\right)$-unique-neighbor expander
for some positive absolute constants $\alpha,\epsilon$.

In response to a question of Tali Kaufman \cite{kaufman_private}, who asked if there is a
infinite family of $\left(\alpha,\epsilon\right)$-unique-neighbor expanders
which are Cayley graphs, we develop a similar graph product in the
context of Cayley graphs. The product of a graph $\Gamma'$ and a graph
$\Delta$ is in general not symmetric, even if both graphs are Cayley
graphs. However, we show that if $\Gamma'$ and $\Delta$ are Cayley
graphs (with respect to the groups $G$ and $H$ respectively), and
if further $\Gamma'$ is bipartite and simply-generator-symmetric
with respect to the group $H$ (i.e. $H$ acts simply transitively
on the generators of $G$ through group automorphisms of $G$), then
the product graph $\Gamma$ can be formed 'symmetrically' and is then
itself a Cayley graph. The graph $\Delta$ in \cite{alon2002explicit}
is indeed a Cayley graph on the cyclic group $C_{8}$. Fortunately,
there are $8$-regular bipartite Ramanujan graphs which are simply-generator-symmetric
with respect to the group $C_{8}$. Such graphs were constructed explicitly
by Lubotzky, Samuels and Vishne in \cite{zbMATH02183844},
as a special case of the construction of Ramanujan complexes. Thus,
we show that the method of \cite{alon2002explicit}, applied in a symmetric manner to this infinite family
of Ramanujan graphs from \cite{zbMATH02183844}, gives rise
to an infinite family of unique-neighbor expanders which are Cayley graphs.
This gives an affirmative answer to Kaufman's question. We conclude:
\begin{thm}
\label{thm:main-une}For some absolute positive constants $\alpha,\epsilon$,
there is an explicit construction of an infinite family of $\left(\alpha,\epsilon\right)$-unique-neighbor expanders,
such that every graph in the family is a Cayley graph.
\end{thm}
Our method has another application: In \cite{zbMATH06294581}, Kaufman
and Lubotzky construct asymptotically good families of symmetric LDPC
error correcting codes. A code $C\subset\mathbb{F}_{2}^{X}$ is \emph{symmetric}
with respect to a group $G$ if there is a transitive group action
of $G$ on $X$, such that the corresponding coordinate-interchanging
action on $\mathbb{F}_{2}^{X}$ preserves $C$. We say that a symmetric
code is \emph{simply-symmetric} if the action of $G$ on $X$ is simply
transitive. The codes in \cite{zbMATH06294581} are based on a product
of a large graph $\Gamma'$ and a small code $B$. This construction
is due to Tanner (\cite{zbMATH03745089}) and Sipser-Spielman (\cite{zbMATH01004587}),
and is referred to as an ``expander code'' in case the graph $\Gamma'$
is a good expander. In case the graph $\Gamma'$ is a Cayley graph,
a more specific construction is defined in \cite{zbMATH05799072}
by Kaufman and Wigderson, and is called a ``Cayley code''.
While Cayley codes are not necessarily symmetric, Kaufman and Wigderson construct
a family symmetric Cayley codes with constant rate, but with normalized distance
$\Omega\left(\frac{1}{\left(\log\log n\right)^2}\right)$, were $n$
is the code length (see Theorem 11 of \cite{zbMATH05799072}).
In \cite{zbMATH06294581},
using the Ramanujan graphs of \cite{zbMATH02183844}, asymptotically
good symmetric Cayley codes are constructed. The construction uses
an infinite family of \emph{bipartite} Ramanujan graphs, but the bipartiteness
is not used in the proof (and similar constructions using non-bipartite
generator-symmetric expanders are possible). We use the bipartiteness
of these Ramanujan graphs, together with the fact that they are \emph{simply}-generator-symmetric,
to show that the codes of \cite{zbMATH06294581} are in fact \emph{simply}
symmetric. In addition to this extra feature, this construction has
the benefit of exhibiting symmetric unique-neighbor expanders and
symmetric error correcting codes in one framework. We conclude:
\begin{thm}
\label{thm:main-codes}There is an infinite family of asymptotically good simply-symmetric
LDPC error correcting codes.
\end{thm}
Finally, we define a variation of the codes of \cite{zbMATH06294581}
which improves the bound on their density from $4094$ to $20$. This
is achieved by a slight variation on the argument of Sipser-Spielman
\cite{zbMATH01004587}.

\section{Edge Transitive Ramanujan Graphs\label{sec:ET-Ramanujan}}

Let $\Gamma=\Cay\left(G,S\right)$ be a Cayley graph where $S$ is a subset
of the group $G$ satisfying $S=S^{-1}$. As a Cayley graph, $\Gamma$ is
automatically vertex-transitive. We say that $\Gamma$
is \emph{generator-symmetric} with respect to the group $H$ and the
group homomorphism $\theta:H\rightarrow\Aut\left(G\right)$
if $\theta\left(H\right)$ preserves $S$ and induces a transitive
action of $H$ on $S$. In this case, the group $G\rtimes_{\theta}H$
acts transitively on the edges of $\Gamma$, and so $\Gamma$ is an
edge-transitive graph. If the action of $H$ on $S$ is simply-transitive
(in other words, if the action is both transitive and free),
then we say that the Cayley graph $\Gamma$ is \emph{simply-generator-symmetric}.
This means, in particular, that $\left|H\right|=\left|S\right|$.

In this case, the action of $G\rtimes_\theta H$ on the \emph{directed} edges of
$\Gamma$ is simply transitive.
However, some elements of $G\rtimes_\theta H$ act on $\Gamma$ by
"reversing an edge".
Therefore, if one regards $\Gamma$ as an undirected graph, the action of
$G\rtimes_\theta H$ on edges, while still transitive, is no longer free.
Our construction in section \ref{sec:Line-Graphs}  is based on the observation
that if $\Gamma$ is a bipartite simply-generator-symmetric Cayley graph,
then there is a subgroup of $G\rtimes_\theta H$ of index $2$ which acts
simply transitively on the $undirected$ edges of $\Gamma$.
For notational convenience, for every $h\in H$, we denote the automorphism
$\theta\left(h\right)\in\Aut\left(G\right)$ by $\theta_{h}$.

In \cite{zbMATH02183844}, Lubotzky, Samuels and Vishne construct
highly symmetric Ramanujan complexes, which we now specialize to the
one dimensional case (i.e. graphs). We get two infinite families of
$d$-regular Ramanujan graphs for each $d=q+1$, where $q$ is an
odd prime power. One is a family of bipartite graphs and the other
is a family of non-bipartite graphs. All of these graphs are Cayley
graphs and are simply-generator-symmetric with respect to the cyclic
group $H=C_{d}$. For a brief overview of this special case see section
4 of \cite{zbMATH06294581}.

Specializing further to the case of bipartite $8$-regular graphs,
we get from \cite{zbMATH02183844} an infinite family of simply-generator-symmetric
bipartite Ramanujan Cayley graphs $\left\{ \Gamma'_{n}\right\} $ with $\Gamma'_{n}=\Cay\left(\PGL_{2}\left(7^{n}\right),S\left(n\right)\right)$
for some set of generators $S\left(n\right)\subset\PGL_{2}\left(7^{n}\right)$
defined in \cite{zbMATH02183844}. The group $\PSL_{2}\left(7^{n}\right)$
is of index $2$ in $\PGL_{2}\left(7^{n}\right)$, and is disjoint
from $S\left(n\right)$. As shown in \cite{zbMATH02183844},
there is a group homomorphism $\theta_{n}:C_{8}\rightarrow\Aut\left(\PGL_{2}\left(7^{n}\right)\right)$
, inducing a simply-transitive action of $C_{8}$ on $S\left(n\right)$.
For an explicit definition of the generators $S\left(n\right)$ and
the action of $C_{8}$, see \cite{zbMATH02183844}, Algorithm
9.2. In particular, Step 3 of this algorithm uses an embedding $C_{8}\cong\mathbb{F}_{7^{2}}^{\times}/\mathbb{F}_{7}^{\times}\hookrightarrow\PGL_{2}\left(7^{n}\right)$
to define the generators $S\left(n\right)$ as the $C_{8}$-orbit
of a certain element $b\in\PGL_{2}\left(7^{n}\right)$, where
$C_{8}$ acts on $\PGL_{2}\left(7^{n}\right)$ by inner automorphisms
through this embedding. Thus, the corresponding homomorphism $\theta_{n}:C_{8}\rightarrow\Aut\left(\PGL_{2}\left(7^{n}\right)\right)$
is as required.

\section{Line Graphs as Cayley Graphs\label{sec:Line-Graphs}}

For a graph $\Gamma'$=$\left(V,E\right)$, the \emph{line graph}
$\Gamma$ of $\Gamma'$ is the graph whose vertex set is $E$, with
vertices $e_{1},e_{2}\in E$ of $\Gamma$ adjacent in $\Gamma$ if
and only if $e_{1}$ and $e_{2}$, as edges of $\Gamma'$, are
incident to a common vertex of $\Gamma'$. In general, the line graph
of a Cayley graph is \emph{not }itself a Cayley graph. In this section,
we show that the line graph of a simply-generator-symmetric \emph{bipartite}
Cayley graph is itself a Cayley graph. Let $G$ be a finite group
and let $S\subset G$ be a symmetric subset of $G$ (i.e. $S=S^{-1}$)
which generates $G$. Let $K$ be a subgroup of $G$ of index $2$
such that $K$ and $S$ are disjoint. Fix a generator $s_{0}\in S$.
Then, the Cayley graph $\Cay\left(G,S\right)$ is bipartite
with the cosets $K$ and $Ks_{0}$ as its left and right sides respectively.
The coset $K$ consists of the vertices connected to $1_{G}$ by a
path of even length, and the coset $Ks_{0}$ consists of the vertices
connected to $1_{G}$ by a path of odd length. Note that for a Cayley
graph $\Cay\left(G,S\right)$, the existence of an index $2$
subgroup $K\leq G$ disjoint from $S$ is in fact equivalent to
$\Cay\left(G,S\right)$ being bipartite. Assume further that
$\Cay\left(G,S\right)$ is simply-generator-symmetric with respect
to the group $H$ and the group homomorphism $\theta:H\rightarrow\Aut\left(G\right)$.
Then, for each element $h\in H$, we have $\theta_{h}\left(K\right)=K$.
Therefore, we can define the semidirect product $K\rtimes_{\theta}H$. Finally, we let $T$ be a subset of $S$.

In this setting, we define two graphs $X$ and $\Gamma$, and show
that they are isomorphic.

First, for every $h\in H$, we define $\sigma_{1}\left(h\right),\sigma_{2}\left(h\right)\in K\rtimes_{\theta}H$
by
\begin{eqnarray*}
\sigma_{1}\left(h\right) & = & \left(1_{K},h\right)\\
\sigma_{2}\left(h\right) & = & \left(s_{0}\cdot\theta_{h}\left(s_{0}^{-1}\right),h\right)
\end{eqnarray*}

and we let $X=\mbox{Cay}\left(K\rtimes_{\theta}H,\Sigma_{1}^{T}\cup\Sigma_{2}^{T}\right)$,
where
\begin{eqnarray*}
\Sigma_{1}^{T} & = & \left\{ \sigma_{1}\left(t\right)\mid t\in T\right\} \\
\Sigma_{2}^{T} & = & \left\{ \sigma_{2}\left(t\right)\mid t\in T\right\}
\end{eqnarray*}

Second, we define a graph $\Gamma=\Gamma\left(G,S,s_0,K,H,\theta,T\right)$.
The vertices of $\Gamma$ are the undirected edges of $\mbox{Cay}\left(G,S\right)$, i.e. $V\left(\Gamma\right)=E\left(\mbox{Cay}\left(G,S\right)\right)$.
For each $g\in G$, we let $E_{g}$ be the set of undirected edges in $\mbox{Cay}\left(G,S\right)$
which are incident to $g$, and we define a bijection $e_{g}:H\rightarrow E_{g}$
by
\[
e_{g}\left(h\right)=\begin{cases}
\left(g,g\cdot\theta_{h}\left(s_{0}\right)\right) & g\in K\\
\left(g\cdot\theta_{h}\left(s_{0}^{-1}\right),g\right) & g\in Ks_{0}
\end{cases}
\]

The edges of $\Gamma$ are $E\left(\Gamma\right)=\left\{ \left(e_{g}\left(h\right),e_{g}\left(ht\right)\right)\mid g\in G,h\in H,t\in T\right\} $.

Note that if $T=T^{-1}$, then $\Gamma$ is an undirected graph, and
if $T=S$, then $\Gamma$ is the line graph of $\mbox{Cay}\left(G,S\right)$.
\begin{prop}\label{prop:line-graph}
The bijection $f:K\rtimes_{\theta}H\rightarrow E\left(\mbox{Cay}\left(G,S\right)\right)$
defined by
\[
f\left(\left(k,h\right)\right)=\left(k,k\cdot\theta_{h}\left(s_{0}\right)\right)
\]
is a graph isomorphism from $X=\mbox{Cay}\left(K\rtimes_{\theta}H,\Sigma_{1}^{T}\cup\Sigma_{2}^{T}\right)$
to $\Gamma=\Gamma\left(G,S,s_0,K,H,\theta,T\right)$.\end{prop}
\begin{proof}
Let $k\in K,h\in H$ and $t\in T$. Direct computation shows that
\begin{eqnarray*}
\left(\left(k,h\right),\left(k,h\right)\cdot\sigma_{1}\left(t\right)\right) & \overset{f\times f}{\longmapsto} & \left(e_{k}\left(h\right),e_{k}\left(ht\right)\right)\\
\left(\left(k,h\right),\left(k,h\right)\cdot\sigma_{2}\left(t\right)\right) & \overset{f\times f}{\longmapsto} & \left(e_{k\cdot\theta_{h}\left(s_{0}\right)}\left(h\right),e_{k\cdot\theta_{h}\left(s_{0}\right)}\left(ht\right)\right)
\end{eqnarray*}
and
\begin{eqnarray*}
\left(e_{k}\left(h\right),e_{k}\left(ht\right)\right) & \overset{f^{-1}\times f^{-1}}{\longmapsto} & \left(\left(k,h\right),\left(k,h\right)\cdot\sigma_{1}\left(t\right)\right)\\
\left(e_{ks_{0}}\left(h\right),e_{ks_{0}}\left(ht\right)\right) & \overset{f^{-1}\times f^{-1}}{\longmapsto} & \left(\left(ks_{0}\theta_{h}\left(s_{0}^{-1}\right),h\right),\left(ks_{0}\theta_{h}\left(s_{0}^{-1}\right),h\right)\cdot\sigma_{2}\left(t\right)\right)
\end{eqnarray*}
Therefore, both $f$ and $f^{-1}$ are graph homomorphisms, and therefore
$f$ is a graph isomorphism.\end{proof}

In particular, Proposition \ref{prop:line-graph} implies that the line graph of a generator symmetric bipartite Cayley graph is itself a Cayley graph.
It may be interesting to characterize the Cayley graphs whose line graph is a Cayley graph.

\section{Symmetric Unique Neighbor Expanders}

We begin by describing the construction of an infinite family of $6$-regular
unique-neighbor expanders by Alon and Capalbo (see \cite{alon2002explicit}).
Let $\Gamma'=\left(V',E'\right)$ be a $d$-regular graph and let
$\Delta=\left(V\left(\Delta\right),E\left(\Delta\right)\right)$ be
a graph on $d$ vertices. For each vertex $v$ of $\Gamma'$, denote
by $E'_{v}$ the set of $d$ edges of $\Gamma'$ which are incident
to $v$. For each vertex $v$ of $\Gamma'$, fix a bijection between
$E'_{v}$ and $V\left(\Delta\right)$. Note that these bijections
\emph{do not }need to be compatible with each other in any way. Form
a new graph $\Gamma$ whose vertex set is $E'$, and where $\gamma_{1},\gamma_{2}\in E'$
are adjacent as vertices of $\Gamma$ if and only if $\gamma_{1}$
and $\gamma_{2}$ are both incident to a common vertex $v$ in $\Gamma'$,
and are neighbors in $\Delta$ under the identification $E'_{v}\leftrightarrow V\left(\Delta\right)$.
Note that the graph $\Gamma$ is a subgraph of the line graph of $\Gamma'$.
By Theorem 2.1 of \cite{alon2002explicit}, if $\Gamma'$ is an $8$-regular
Ramanujan graph, and if $\Delta$ is the $3$-regular graph on $8$
vertices $v_{0},\dotsc,v_{7}$ with $v_{i}$ adjacent to $v_{i-1},v_{i+1},v_{i+4}$
(indices taken modulo $8$), then $\Gamma$ is a $6$-regular $\left(\alpha,1/10\right)$-unique-neighbor
expander, for some positive constant $\alpha$.

In this way, an infinite family of (not necessarily bipartite) $8$-regular
Ramanujan graphs gives rise to an infinite family of $6$-regular
$\left(\alpha,1/10\right)$-unique-neighbor expanders. In \cite{alon2002explicit},
the chosen family of Ramanujan graphs is the one constructed in \cite{zbMATH04079458}
by Lubotzky, Phillips and Sarnak. These Ramanujan graphs are Cayley
graphs and thus are vertex-transitive, but they are not known to be
generator-symmetric.

Instead, we use the $8$-regular simply-generator-symmetric bipartite Ramanujan graphs of \cite{zbMATH02183844}.
Keeping the notation of sections \ref{sec:ET-Ramanujan} and \ref{sec:Line-Graphs}, we define for each $n\geq 1$,
\[
\Gamma_{n}=\Gamma\left(\PGL_{2}\left(7^{n}\right),S\left(n\right),b,\PSL_{2}\left(7^{n}\right),C_{8},\theta_{n},\left\{ 1,4,7\right\}\right)
\]

On one hand, the graphs $\{\Gamma_n\}$ are a special case of the Alon-Capalbo construction. On the other hand,
by Proposition \ref{prop:line-graph}, for each $n\geq1$, the graph $\Gamma_n$ is isomorphic to
$\mbox{Cay}\left(\PSL_{2}\left(7^{n}\right)\rtimes_{\theta_n}C_{8},\Sigma_{1}^{\left\{ 1,4,7\right\}}\cup\Sigma_{2}^{\left\{ 1,4,7\right\}}\right)$.

We have thus proved the following more elaborate version of Theorem \ref{thm:main-une}:
\begin{thm}
For some constant $\alpha>0$, there is an infinite family $\left\{ \Gamma_{n}\right\} $
of $6$-regular $\left(\alpha,1/10\right)$-unique-neighbor expanders,
such that for every $n$, the graph $\Gamma_{n}$ is a Cayley graph
on the group $\PSL_{2}$$\left(7^{n}\right)\rtimes_{\theta_{n}}C_{8}$.
\end{thm}

\section{Symmetric Good LDPC Codes}

We shall refer to linear error correcting codes simply as 'codes'.
A code $C\subset\mathbb{F}_{2}^{X}$ is \emph{symmetric} (resp. \emph{simply-symmetric})
with respect to a group $G$ if there is a transitive (resp. \emph{simply-transitive})
action of $G$ on $X$ such that the corresponding coordinate-interchanging
action of $G$ on $\mathbb{F}_{2}^{X}$ preserves $C$. For a code
$C\subset\mathbb{F}_{2}^{X}$, we denote its dual by $C^{\perp}\subset\mathbb{F}_{2}^{X}$
(this is the set of all vectors ``orthogonal'' to $C$), and think
of it as the constraints defining the code $C$ (i.e. they define
linear functionals whose common set of solutions is $C$). A spanning
set for $C^{\perp}$ is called a \emph{set of defining constraints}.
By a standard abuse of terminology, we will refer to a code $C$ together
with a specific set of defining constraints simply as a 'code'. Note
that if $C$ is a symmetric code with respect to a group $G$, then
the dual code $C^{\perp}$ is also symmetric with respect to $G$.
Finally, a family of codes is \emph{LDPC} if the defining constraints
of the codes in the family are of bounded Hamming weight.

One way to obtain symmetric codes is by 'codes defined on groups':
For a group $G$, we say that a code $C\subset\mathbb{F}_{2}^{G}$
is a \emph{code defined on $G$} if $C$ is invariant under the action
of $G$. Equivalently, a code $C\subset\mathbb{F}_{2}^{G}$ is a code
defined on $G$ if it is defined by a $G$-invariant set of constraints.
Finally, a code $C$ is defined on the group $G$ if and only if it
is simply-symmetric with respect to $G$.

Another way to obtain (usually not symmetric) codes is by 'codes defined
on graphs': Let $\Gamma=\left(V,E\right)$ be an $l$-regular graph
and let $B\subset\mathbb{F}_{2}^{H}$ be a 'small' code of length
$l$ (i.e. $H$ is a set of cardinality $l$). For each vertex $v$
of $\Gamma$, fix a bijection $h\mapsto e\left(v,h\right)$ from $H$
to $E_{v}$. Let $C\subset\mathbb{F}_{2}^{E}$ be the code consisting
of functions $f:E\rightarrow\mathbb{F}_{2}$ satisfying the following
local constraint for each vertex $v$ of $\Gamma$:
\[
\left(f\left(e\left(v,h\right)\right)\right)_{h\in H}\in B
\]
These local constraints are referred to as 'vertex consistency'. This
idea is due to Tanner (\cite{zbMATH03745089}) and Sipser-Spielman (\cite{zbMATH01004587}),
who refer to such codes (and similar codes) as \emph{expander codes}
in case the graph $\Gamma$ is a good expander. If the graph $\Gamma$
is a Cayley graph $\Gamma=\Cay\left(G,S\right)$ then, instead
of arbitrarily fixing a bijection between $H$ and $E_{v}$ for every
vertex $v$ of $\Gamma$, we can fix one bijection $h\mapsto s_{h}$
from $H$ to $S$. Then, for each vertex $v$ of $\Gamma$, we define
the bijection from $H$ to $E_{v}$ by $e\left(v,h\right)=\left(v,v\cdot s_{h}\right)$
(where $\cdot$ is multiplication in the group $G$), and proceed
as before to define the code $C$. Codes defined on Cayley graphs
in this manner are referred to as \emph{Cayley codes} by Kaufman and
Wigderson in \cite{zbMATH05799072} and are denoted by $\Cay\left(G,S,B\right)$.

As stated in \cite{zbMATH06294581}, Cayley codes are in general
\emph{not }symmetric\emph{. }Therefore, we assume further that $H$
is a group (and not merely a set), that the code $B$ is a code defined
on $H$, and that the Cayley graph $\Cay\left(G,S\right)$ is
simply-generator-symmetric with respect to the group $H$. Note that
in this case we have $\left|H\right|=\left|S\right|$. Let $\theta:H\rightarrow\Aut\left(G\right)$
be the group homomorphism inducing the transitive action of $H$ on
$S$. Now, instead of arbitrarily fixing a bijection between $H$
and $S$, we only fix one generator $s_{0}\in S$. We define the bijection
from $H$ to $S$ by $h\mapsto\theta_{h}\left(s_{0}\right)$, and
proceed as before to define the code $C$. This idea is due to Kaufman
and Lubotzky and is presented in \cite{zbMATH06294581}. The resulting
code $C$ is symmetric with respect to the group $G\rtimes_{\theta}H$.
However, the action of this group on the code $C$ is not \emph{simply
}transitive (indeed, the cardinality of the group $G\rtimes_{\theta}H$
is not equal to the length of the code $C$; the size of the group
is twice the length of the code). Therefore, $C$ is not a code defined
on the group $G\rtimes_{\theta}H$.

However, if the graph $\Cay\left(G,S\right)$ is \emph{bipartite},
then there is an index $2$ subgroup $K$ of $G$, disjoint from $S$.
Then, by Section \ref{sec:Line-Graphs}, the code $C$ is a code defined
on $K\rtimes_{\theta}H$ (and not just a symmetric code with respect
to $G\rtimes_{\theta}H$, as in \cite{zbMATH06294581}). If the
small code $B$ is defined by $m$ $H$-orbits of constraints, then
the code $C$ is defined by $2m$ $K\rtimes_{\theta}H$-orbits of
constraints (or $m$ $G\rtimes_{\theta}H$-orbits of constraints,
as in \cite{zbMATH06294581}). We summarize this result in the case
where the small code $B$ is defined by a single orbit of constraints
(since this is sufficient for our application):
\begin{prop}
\label{prop:group-codes}For $G,S,s_{0},K,H,\theta$ as in section \ref{sec:Line-Graphs}, and a 'small' code $B$ defined on the group $H$ by the orbit of the single constraint $\sum_{t\in T}f\left(t\right)=0$
for some $T\subset H$, the Cayley code $\Cay\left(G,S,B\right)$
is the same as the code defined on the group $K\rtimes_{\theta}H$ with defining
constraints consisting of the $K\rtimes_{\theta}H$-orbits of the
two constraints $\sum_{t\in T}f\left(\left(1,t\right)\right)=0$ and
$\sum_{t\in T}f\left(\left(s_{0}\cdot\theta_{t}\left(s_{0}^{-1}\right),t\right)\right)=0$.
\end{prop}
Finally, note that in \cite{zbMATH06294581} the Cayley code construction
is applied with a sequence $\left\{ \Gamma'_{n}\right\} $ of $q+1$-regular
Ramanujan graphs from \cite{zbMATH02183844}, where $q=4093$,
and a certain code $B_{0}$ defined on $C_{q+1}$. Each graph $\Gamma'_{n}$
in the sequence is a Cayley graph on $\PGL_{2}\left(q^{n}\right)$
and is simply-generator-symmetric with respect to $C_{q+1}$. These
graphs are in fact bipartite, although the proofs in \cite{zbMATH06294581}
do not require bipartiteness. Therefore, by Proposition \ref{prop:group-codes},
we conclude that the codes constructed in \cite{zbMATH06294581}
are in fact simply-symmetric. These codes are proved in \cite{zbMATH06294581}
to have both rate and normalized distance bounded away from zero (i.e.
they are \emph{asymptotically good} codes). Also, they clearly are
LDPC codes. We have thus proved the following more elaborate version
of Theorem \ref{thm:main-codes}:
\begin{thm}
\label{thm:main-codes-elaborate}There is an asymptotically good infinite family
of LDPC codes $\left\{ K_{n}\right\} $, such that for every $n$,
the code $K_{n}$ is simply-symmetric with respect to the group $G_{n}=\PSL_{2}\left(q^{n}\right)\rtimes_{\theta_{n}}C_{q+1}$,
where $q=4093$. Furthermore, for every $n$, the code $K_{n}$ is
defined by constraints consisting of two $G_{n}$-orbits of constraints.
\end{thm}

\section{Symmetric Good LDPC Codes with Improved Density}

If $C$ is a code (together with a set of defining constraints), then
we define the \emph{density} of $C$ as the maximal Hamming weight
of a defining constraint of $C$. The density of the family of codes
$\left\{ K_{n}\right\} $ is proved in \cite{zbMATH06294581} to
have density bounded from above by $4094$. We shall define a variation
of this construction, with density $20$. The density of an expander
code on a graph $\Gamma$ and a small code $B$ equals the density
of the small code $B$. Thus, our goal is to define an infinite family of simply-symmetric
expander codes using a small code with density $20$. We shall define
later a code $B'$ of length $158$, rate $\frac{40}{79}$, distance
$15$ and density $20$. Let $\left\{ K'_{n}\right\} $ be the family
of expander codes constructed symmetrically using the $158$-regular
simply-generator-symmetric Ramanujan graphs of \cite{zbMATH02183844}
and the small code $B'$. The family $\left\{ K'_{n}\right\} $ is
a family of simply-symmetric codes of density $20$. We now show that
this is a family of asymptotically good codes.

By Lemma 15 of \cite{zbMATH01004587}, an expander code on a $k$-regular
graph $\Gamma$ and a small code $B$ has rate at least $2\cdot\rate\left(B\right)-1$
and normalized distance at least $\left(\frac{\distance\left(B\right)-\lambda}{k-\lambda}\right)^{2}$
where $\lambda$ is the second largest eigenvalue of the adjacency
matrix of $\Gamma$ (assuming $\distance\left(B\right)>\lambda$).
Thus, since $\rate\left(B'\right)=\frac{40}{79}>\frac{1}{2}$,
we have a guarantee that the codes $\left\{ K'_{n}\right\} $ have
rate bounded away from zero. However, since $\distance\left(B'\right)=15$,
and since $2\sqrt{157}\approx25.1$, the method of \cite{zbMATH01004587}
is not enough to show that the codes $\left\{ K'_{n}\right\} $ have
normalized distance bounded away from zero. Nevertheless, we do have
$15>1+\sqrt{157}\approx13.5$, and thus the next lemma shows that
the codes $\left\{ K'_{n}\right\} $ do have distance bounded away
from zero:
\begin{lem}
An expander code $C$ defined on a $k$-regular Ramanujan graph $\Gamma$
using a 'small' code $B$ of length $k$ and distance $d$ larger
than $1+\sqrt{k-1}$ has normalized distance larger than some constant
$\alpha$ which depends only on $k$ and $d$.\end{lem}
\begin{proof}
Denote $\Gamma=\left(V,E\right)$. By Theorem 4.2 of \cite{zbMATH01103555},
for every $\delta>0$, and every nonempty subset $X$ of $V$ of cardinality
at most $k^{-1/\delta}\cdot\left|V\right|$, the average degree of
the induced subgraph of $\Gamma$ on $X$ is at most $\left(1+\sqrt{k-1}\right)\cdot\left(1+C\delta\right)$,
where $C>0$ is an absolute constant (see also \cite{alon2002explicit},
Theorem 2.2). Let $\epsilon>0$ be sufficiently small such that $1+\sqrt{k-1}+\epsilon<d$,
let $\delta=\frac{\epsilon}{C\left(1+\sqrt{k-1}\right)}>0$, and let
$\alpha=k^{-1/\delta-1}>0$. Let $f:E\rightarrow\mathbb{F}_{2}$
be a nonzero word of Hamming weight $w$ such that $w\leq\alpha\cdot\left|E\right|$.
Let $S$ be the set of edges $e$ of $\Gamma$ such that $f\left(e\right)=1$.
Let $T$ be the subset of $V$ of all vertices in $\Gamma$ which
are incident to at least one edge in $S$. Let $\Gamma\left[T\right]$
be the subgraph of $\Gamma$ induced on $T$. Since $0<\left|T\right|\le2\cdot\left|S\right|\leq2\alpha\cdot\left|E\right|=\alpha\cdot k\cdot\left|V\right|=k^{-1/\delta}\cdot\left|V\right|$,
the average degree in $\Gamma\left[T\right]$ is at most $\left(1+\sqrt{k-1}\right)\cdot\left(1+C\cdot\delta\right)=\left(1+\sqrt{k-1}\right)\cdot\left(1+\frac{\epsilon}{1+\sqrt{k-1}}\right)=1+\sqrt{k-1}+\epsilon<d$.
Therefore, $\Gamma\left[T\right]$ has a vertex $v$ of degree less
than $d$. The local subword of $f$ around the vertex $v$ is of
positive Hamming weight smaller than $d$ and thus violates a constraint.
Therefore, $f$ is not a codeword of $C$.
\end{proof}
The improvement in this lemma over Lemma 15 of \cite{zbMATH01004587}
boils down to the invocation of Kahale's Theorem 4.2 in \cite{zbMATH01103555}
instead of the Alon-Chung Lemma (\cite{zbMATH04073036}, Lemma 2.3). This
is the same tool that enables the proof of unique-neighbor expansion
in \cite{alon2002explicit}.

Thus, we get the following improved version of Theorem \ref{thm:main-codes-elaborate}:
\begin{thm}
\label{thm:main-codes-elaborate-improved-density}There is an asymptotically
good infinite family of LDPC codes $\left\{ K'_{n}\right\} $, such that for
every $n$, the code $K'_{n}$ is simply-symmetric with respect to
the group $G'_{n}=\PSL_{2}\left(q^{n}\right)\rtimes_{\theta_{n}}C_{q+1}$,
where $q=157$. Each code $K'_{n}$ has density $20$, and is defined
by constraints consisting of two $G'_{n}$-orbits of constraints.
\end{thm}
Note that the improvement in density over the symmetric good LDPC codes of \cite{zbMATH06294581}
comes on the expense of a lower guaranteed normalized distance.

We now define a code $B'$ with the desired properties. See Chapter
6 of \cite{zbMATH00053945} for the theory of cyclic codes. Let $B''$ be
the cyclic code of length $79$ generated by
\begin{eqnarray*}
g\left(X\right) & = & X^{39}+X^{36}+X^{35}+X^{31}+X^{30}+X^{29}+X^{27}+X^{26}+X^{25}+X^{24}+X^{21}+\\
 &  & X^{20}+X^{19}+X^{18}+X^{16}+X^{14}+X^{13}+X^{11}+X^{5}+X^{4}+X^{2}+X+1
\end{eqnarray*}

By \cite{zbMATH03593474}, the code $B''$ has rate $\frac{40}{79}$
and distance $15$. Direct computation shows that $h\left(X\right)=\left(X^{79}-1\right)/g\left(X\right)$
has exactly $20$ nonzero coefficients. Thus, the density of the code
$B''$ is $20$ (see Section 6.2 of \cite{zbMATH00053945}). Let $B'$ be
the cyclic code of length $158=2\cdot79$ generated by $g\left(X\right)^{2}$.
Note that the codewords of $B'$ are exactly the result of interleaving
two codewords of $B''$. The code $B'$ has rate $\frac{40}{79}$,
distance $15$ and density $20$.

\section{Conclusion}

We make use of of the simple-generator-symmetry of the Ramanujan graphs
of \cite{zbMATH02183844}, and of our observation that the
line graph of a simply-generator-symmetric \emph{bipartite} Cayley graph
is itself a Cayley graph. This observation is crucial in the realization
of some of the unique-neighbor expanders of \cite{alon2002explicit} as
Cayley graphs. It also allows us to see that the asymptotically good
symmetric codes of \cite{zbMATH06294581} are in fact simply-symmetric. Finally, we improve the density of these codes.

\section{Acknowledgments}
This work was done as part of an M.Sc. thesis submitted to the Hebrew University
of Jerusalem. The author would like to thank his advisor Prof. Alexander
Lubotzky for his help and guidance. The author would also like to thank ETH
Institute for Theoretical Studies for the hospitality.

\section*{References}

\bibliography{mybibfile}

\end{document}